\documentclass{article}
\usepackage[utf8]{inputenc}
\usepackage{amsthm,amsmath,amssymb,xcolor,xparse}
\usepackage{tikz}
\usepackage{multicol}
\usepackage{authblk}
\usetikzlibrary{patterns}
\usepackage{xparse}

\makeatletter
\DeclareRobustCommand{\thmprime}{
 \begingroup
 \expandafter\in@\expandafter b\expandafter{\f@series}
 \ifin@ \boldmath \fi
 $\m@th{}^{\prime}$
 \endgroup
}
\makeatother

\newtheorem{theorem}{Theorem}
\newtheorem{conj}{Conjecture}
\newtheorem{cor}{Corollary}

\newtheorem{lemma}{Lemma}
\newtheorem{prop}{Proposition}
\theoremstyle{definition}
\newtheorem*{defn}{Definition}

\NewDocumentEnvironment{manual}{O{obs}m}
 {
 \addtocounter{obs}{-1}
 \renewcommand\theobs{#2}
 \begin{#1}
 }
 {\end{#1}}
 
 \NewDocumentEnvironment{manual2}{O{lemma}m}
 {
 \addtocounter{lemma}{-1}
 \renewcommand\thelemma{#2}
 \begin{#1}
 }
 {\end{#1}}
 
 \NewDocumentEnvironment{manual3}{O{cor}m}
 {
 \addtocounter{cor}{-1}
 \renewcommand\thecor{#2}
 \begin{#1}
 }
 {\end{#1}}

\newcommand\dist{\operatorname{Dist}}

\newcommand{\gm}{\mu_t(G)}

\begin{document}

\title{Distinguishing Generalized Mycielskian Graphs}

\author{Debra Boutin\thanks{{\tt dboutin@hamilton.edu}, Hamilton College, Clinton, NY}, Sally Cockburn \thanks{{\tt scockbur@hamilton.edu}, Hamilton College, Clinton, NY}, Lauren Keough \thanks{{\tt keoulaur@gvsu.edu}, Grand Valley State University, Allendale Charter Township, MI}, Sarah Loeb \thanks{{\tt sloeb@hsc.edu}, Hampden-Sydney College, Hampden-Sydney, VA}, K.~E. Perry \thanks{{\tt kperry@soka.edu}, Soka University of America, Aliso Viejo, CA}, Puck Rombach \thanks{{\tt puck.rombach@uvm.edu}, University of Vermont, Burlington, VT}}

\date{\today}

\maketitle

\begin{abstract}
A graph $G$ is \emph{$d$-distinguishable} if there is a coloring of the vertices with $d$ colors so that only the trivial automorphism preserves the color classes. The smallest such $d$ is the \emph{distinguishing number}, $\dist(G)$. The \emph{Mycielskian} of a graph $G$, $\mu(G)$, is constructed by adding a shadow vertex $u_i$ for each vertex $v_i$ of $G$, one additional vertex $w$, and  edges so that $N(u_i)~=~N_G(v_i)~\cup~\{w\}$. The \emph{generalized Mycielskian}, $\gm$, is a Mycielskian graph with $t$ layers of shadow vertices, each with edges to layers above and below. This paper examines the distinguishing number of the traditional and generalized Mycielskian graphs. Notably, if $G~\neq ~K_1,~K_2$ and 
the number of isolated vertices in $\gm$ is at most $\dist(G)$, then $\dist(\gm) \le \dist(G)$. This result proves and exceeds a conjecture of Alikhani and Soltani.
\end{abstract}

\section{Introduction}

Vertex colorings can be a used to study the symmetries of a graph, whether or not the automorphism group of the graph is explicitly known. In this paper we study vertex colorings that are not preserved under any nontrivial automorphism. Such colorings are said to be \emph{distinguishing}. The necessary (and sufficient) feature of a distinguishing coloring is that every vertex in the graph can be uniquely identified by its graph properties and its color.

More formally, a coloring of the vertices of a graph $G$ with the colors $1,\ldots, d$ is called a \emph{$d$-distinguishing coloring} if no nontrivial automorphism of $G$ preserves the color classes. 
The distinguishing number of $G$,  $\dist(G)$, is the least $d$ such that $G$ has a $d$-distinguishing coloring. Albertson and Collins introduced graph distinguishing in \cite{AC1996}. In 1977 \cite{Ba1977}, Babai independently introduced the same definition, calling it an \emph{asymmetric coloring}.  In this paper, we will use the terminology of Albertson and Collins. There has been an increasing amount of interest in graph distinguishing since its introduction. 

Much of the work in the last few decades has dealt with large families of graphs, producing results that frequently show that all but a finite number of graphs in the family have distinguishing number $2$. Examples of such families of finite graphs include: hypercubes $Q_n$ with $n\geq 4$ \cite{BC2004}, Cartesian powers $G^n$ for a connected graph $G\ne K_2,K_3$ and $n\geq 2$ \cite{A2005, IK2006,KZ2007}, Kneser graphs $K_{n:k}$ with $n\geq 6, k\geq 2$ \cite{AB2007}, and (with seven small exceptions) $3$-connected planar graphs \cite{FNT2008}. Examples of such families of infinite graphs include: the denumerable random graph \cite{IKT2007}, the infinite hypercube \cite{IKT2007}, locally finite trees with no vertex of degree 1 \cite{WZ2007}, and denumerable vertex-transitive graphs of connectivity 1 \cite{STW2012}. 

Each of the Mycielskian and generalized Mycielskian constructions was introduced to build increasingly large graphs with a given fixed property, but with increasing chromatic numbers. In \cite{M1955}, Mycielski introduced his (traditional) construction, denoted $\mu(G)$, to build from a triangle-free graph $G$ another triangle-free graph with larger chromatic number. Similarly, the generalized Mycielskian construction with $t$ levels, denoted $\mu_t(G)$, was defined to build from a graph $G$ with no small odd cycles, another graph with no small odd cycles, but with   larger chromatic number.  This generalized construction was introduced by Stiebitz~\cite{stiebitz1985beitrage} in 1985 (cited in \cite{ct2001}) and independently by Van Ngoc~\cite{van1987graph} in 1987 (cited in~\cite{van19954}).  Generalized Mycielskian graphs are also called cones over graphs.  Both constructions are formally defined in Section \ref{sec:autos}.

Thus, by design, the traditional and generalized Mycielskian constructions fix some graph parameters while increasing others.  This makes Mycielskian graphs useful for testing and proving  relationships between graph parameters. Recently, there has been significant work studying the effect of these constructions on a variety of vertex parameters. See for example \cite{AbRa2019, BaRa2015, ChXi2006, FiMcBo1998, LWLG2006, PaZh2010} in which various parameters for $\mu(G)$ and $\gm$ are found in terms of the same parameters for $G$.

In this paper we investigate the relationships between the distinguishing numbers of $G$, $\mu(G)$ and $\gm$, for simple graphs $G$. We do this by exploiting the structural properties of $G$ that are inherited when the Mycielskian and generalized Mycielskian constructions are applied. In 2018 \cite{AS2018}, Alikhani and Soltani compared the distinguishing number of $\mu(G)$ to the distinguishing number of $G$ for twin-free $G$. Letting $N(v)$ denote the set of neighbors of $v$, two vertices $x$ and $y$ are called \emph{twins} if $N(x)=N(y)$. A graph having no twins is said to be \emph{twin-free}. For example, vertices $v_1$, $v_2$, and $v_3$ in Figure~\ref{fig:muK13x2} are mutually twin vertices; so are $u_1, u_2,$ and $u_3$. If two vertices of a graph $G$ are twins, then there is an automorphism of $G$ that  exchanges them and fixes the remaining vertices. Thus, a distinguishing coloring must give distinct colors to each vertex in a set of mutual twins. Alikhani and Soltani proved that if $G$ has at least two vertices and is twin-free, then $\dist(\mu(G)) \le \dist(G)+1$.  They then conjectured the following.

\begin{conj}\cite{AS2018} \label{conj:A&S}
 Let $G$ be a connected graph of order $n \ge 3$. Then $\dist(\mu(G))$ $\leq$ $\dist(G)$ except for a finite number of graphs.
\end{conj}

In Theorem~\ref{thm:A&Sconj} (Section \ref{sec:dist}), we prove a statement that is slightly stronger than the above conjecture. 
In particular, we show the conjecture is true for all graphs on at least 3 vertices; not only connected graphs. 
We extend the result to generalized Mycielskians by proving $\dist(\gm)\leq \dist(G)$, unless $G = K_1$, $G=K_2$ and $t=1$, or the number of isolates in $\gm$ exceeds $\dist(G)$. In the last case, $\dist(\gm)$ is exactly the number of isolated vertices.

The paper is organised as follows. The definition of the Mycielskian of a graph $G$, and lemmas regarding automorphisms of $\mu(G)$, are covered in Section~\ref{sec:autos}. The same topics for the generalized Mycielskian of $G$ are developed in Section~\ref{sec:gen}. Theorem~\ref{thm:A&Sconj} on the distinguishing number of $\mu(G)$ and $\gm$ is stated and proved in Section~\ref{sec:dist}. 

In this paper, all graphs are finite simple graphs. We will denote the number of vertices of $G$ by $|G|$ and the degree of a vertex $v$ by $d(v)$.

\section{Mycielskian Graphs} \label{sec:autos}

In this section, we define and examine the traditional Mycielski construction. Suppose $G$ is a graph with $V(G) = \{v_1,\ldots,v_n\}$. The \emph{Mycielskian of $G$}, denoted $\mu(G)$, has vertices $\{v_1,\ldots,v_n,u_1,\ldots,u_n,w\}$. For each edge $v_i v_j$ in $G$, the graph $\mu(G)$ has edges $v_i v_j, v_i u_j$, and $u_i v_j$. In addition, $\mu(G)$ has edges $u_i w$ for $i \in \{1,\ldots,n\}$. Thus, $\mu(G)$ has an isomorphic copy of $G$ on vertices $\{v_1,\ldots,v_n\}$. We refer to vertices from $\{u_1,\ldots,u_n\}$ as \emph{shadow vertices} and vertices from $\{v_1,\ldots,v_n\}$ as \emph{original vertices}. The vertex $w$ that dominates the shadow vertices is called the \emph{root}. As an example, $\mu(K_{1,3})$ is shown in Figure \ref{fig:muK13x2}.

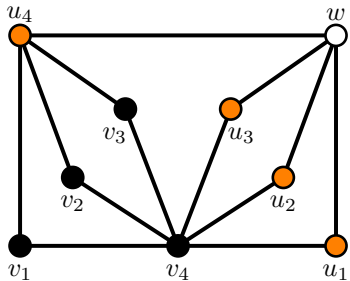
\begin{figure}[ht]
  \centering
  
\begin{tikzpicture}[scale=.7]
\draw[black!100,line width=1.5pt] (-3,0) -- (3,0); 
\draw[black!100,line width=1.5pt] (-3,4) -- (3,4); 
\draw[black!100,line width=1.5pt] (-3,0) -- (-3,4); 
\draw[black!100,line width=1.5pt] (3,0) -- (3,4); 
\draw[black!100,line width=1.5pt] (0,0) -- (2,1.3); 
\draw[black!100,line width=1.5pt] (0,0) -- (1,2.6); 
\draw[black!100,line width=1.5pt] (3,4) -- (2,1.3); 
\draw[black!100,line width=1.5pt] (3,4) -- (1,2.6);
\draw[black!100,line width=1.5pt] (0,0) -- (-2,1.3); 
\draw[black!100,line width=1.5pt] (0,0) -- (-1,2.6); 
\draw[black!100,line width=1.5pt] (-3,4) -- (-2,1.3); 
\draw[black!100,line width=1.5pt] (-3,4) -- (-1,2.6);
\draw[fill=black!100,line width=1] (0,0) circle (.2);
\draw[fill=black!100,line width=1] (-3,0) circle (.2);
\draw[fill=orange!100,line width=1] (3,0) circle (.2);
\draw[fill=white!100,line width=1] (3,4) circle (.2);
\draw[fill=orange!100,line width=1] (-3,4) circle (.2);
\draw[fill=orange!100,line width=1] (2,1.3) circle (.2);
\draw[fill=orange!100,line width=1] (1,2.6) circle (.2);
\draw[fill=black!100,line width=1] (-2,1.3) circle (.2);
\draw[fill=black!100,line width=1] (-1,2.6) circle (.2);
\draw (0,-.5) node{$v_4$};
\draw (-3,-.5) node{$v_1$};
\draw (3,-.5) node{$u_1$};
\draw (3,4.4) node{$w$};
\draw (-3,4.4) node{$u_4$};
\draw (-2,.8) node{$v_2$};
\draw (-1.2,2.1) node{$v_3$};
\draw (2,.8) node{$u_2$};
\draw (1.2,2.1) node{$u_3$};
\end{tikzpicture}
  \caption{The graph $\mu(K_{1,3}).$ The vertices labeled $v_i$ are from $K_{1,3}$, the vertices labeled $u_i$ are the shadow vertices, and $w$ is the root.} \label{fig:muK13x2}
\end{figure}

We will employ the following properties of $\mu(G)$ and automorphisms throughout our proofs.

\medskip\noindent{\bf \boldmath Facts about $\mu(G)$:}  Let $|G|=n$ and $d_G(v_i)=k$. With the notation given above, the Mycielski construction gives us the following: $|\mu(G)|=2n+1$; $d_{\mu(G)}(w) = n$; $d_{\mu(G)}(v_i) = 2k$; $d_{\mu(G)}(u_i)= k+1$; $N_{\mu(G)}(u_i)\setminus\{w\} = N_{G}(v_i)$; $N_{\mu(G)}(w)$ is an independent set (consisting of all shadow vertices).

For the remainder of this paper, when its use is unambiguous, we will drop the subscript $\mu(G)$ from neighborhoods and degrees. That is, unless otherwise noted, for all $x\in V(\mu(G))$, $N(x)=N_{\mu(G)}(x)$ and $d(x)= d_{\mu(G)}(x)$.

\medskip\noindent{\bf \boldmath Facts about Automorphisms of $G$:} Let $\phi$ be an automorphism of a graph $G$ and let $x,y\in V(G)$. Since automorphisms preserve adjacency and nonadjacency of vertex pairs, every property involving adjacency or nonadjacency is also preserved. In particular, degrees: $d(x)=d(\phi(x))$; distances: $d(x,y) = d(\phi(x),\phi(y))$; neighborhoods: ${N(x)} ={N(\phi(x))}$.

First we prove that if there is an automorphism of $\mu(G)$ such that the image of $w$ is an original vertex, then $G$ has no dominating vertex.

\begin{lemma}\label{lem:nodomvertex}
Let $G$ be a graph with $|G|\geq 3$ and let $\phi$ be an automorphism of $\mu(G)$. If $\phi(w)$ is an original vertex, then $G$ cannot have a dominating vertex.
\end{lemma}
\begin{proof}
Let $|G|=n$ and assume $\phi(w)=v$ with $d_G(v)=k$. Using facts about $\mu(G)$ and automorphisms we have $d(w)=n$ and so $d(\phi(w))=d(v) = n$. By construction, $d(v) = 2d_G(v) = 2k$, we have $n=2k$.

Since $n\geq 3$ and $n=2k$, we get $k\geq 2$. So $d_G(v)=k=\frac{n}{2}<n{-}1$ and thus $v$ is not dominating in $G$. Thus, any dominating vertex of $G$ must be in $N(v)$. However, as $N(w)$ is independent, so is $N(\phi(w))=N(v)$. Since $d(v) = k \ge 2$, we conclude $G$ has no dominating vertex in $N(v)$, nor thus in $G$.\end{proof}

We now show that, in fact, any automorphism of $\mu(G)$ that does not fix the root $w$ must map it to a shadow vertex.

\begin{lemma} \label{lem:phiwisv}
Let $G$ be a graph with $|G| \ge 3$. Then no automorphism of $\mu(G)$ maps the root $w$ to any original vertex.
\end{lemma}

\begin{proof}
Let $G$ be a graph with $n \ge 3$ vertices and suppose by way of contradiction that $G$ has an automorphism $\phi$ with $\phi(w) = v$ for some original vertex $v$. We will show that there is no possible image for the shadow of $v$ under $\phi$.

Label the vertices of $G$ so that $v=v_n$ and the neighbors of $v$ in $G$ are $\{v_1,\ldots, v_k\}$ with $k < n$. The shadow vertex of $v$ will then be denoted $u = u_n$. 

Since $u$ is a shadow vertex, it is adjacent to $w$ by construction, and so $\phi(u)$ is adjacent to $\phi(w)=v$. Thus, $\phi(u) \in N(v) = \{v_1,\ldots,v_k,u_1,\ldots,u_k\}$. We consider two cases: $\phi(u) = u_i$ for some $1 \le i \le k$ or $\phi(u) = v_i$ for some $1 \le i \le k$ and find a contradiction in each.

\medskip
\noindent\textbf{Case\,(I):} Suppose that $\phi(u) = u_i$ for some $1 \le i \le k$. 

We will show that this implies $G$ has a dominating vertex, contradicting Lemma~\ref{lem:nodomvertex}.

Since $d(w)=n$ and automorphisms preserve degree, $d(\phi(w))=d(v) = n$ as well. By construction of $\mu(G)$, we have $d(v) = 2d_G(v) = 2k$. Thus, $n=2k$.

 Since $d_G(v)=k$, by construction $d(u) = k+1$. Further, since automorphisms preserve degree, $d(\phi(u))=d(u_i) = k+1$ as well. Since $u_i$ is the shadow vertex of $v_i$, by construction we also get $d(v_i)=2k$. Also, by our choice of $i$, $v_i\in N(v)$. Thus, by properties of the automorphism $\phi^{-1}$, we have $\phi^{-1}(v_i) \in N(\phi^{-1}(v)) = N(w)$. Hence, $w$ has a neighbor of degree $2k$.
 
 Since the only neighbors of the root are shadow vertices, there is some $j$ such that $d(u_j)=2k$. By construction, this means that $d(v_j) = 2(2k{-}1) = 2n{-}2$ and so $d_G(v_j) = n{-}1$. This implies $v_j$ is dominating in $G$, contradicting Lemma~\ref{lem:nodomvertex}. Thus, $\phi(u) \ne u_i$ for any $1\leq i \leq k$.

\medskip
\noindent\textbf{Case\,(II):} Suppose that $\phi(u) = v_i$ for some $1 \le i \le k$. 

We will show that $\phi(v) = w$ and use this to argue that $d(u) = 2$, a contradiction since $d(u) = k+1$ and $k\geq 2$.

Since $N(u) = \{v_1, \dots, v_k, w\}$, 
 we have that \[ N(\phi(u)) = N(v_i) = \{\phi(v_1), \dots, \phi(v_k), \phi(w)=v\}.\]
Since $v_i$ is an original vertex, its neighbors come in original-shadow vertex pairs. In particular, since $v$ is  neighbor of $v_i$, its shadow $u$ must also be a neighbor of $v_i$, which implies that 
$u \in \{\phi(v_1), \dots, \phi(v_k)\} \subset  N(\phi(v)).$
If $u \in N(\phi(v))$, then reciprocally, $\phi(v) \in N(u) = \{v_1, \dots, v_k, w\}$. However, since $v$ is not adjacent to $w$, $\phi(v)$ is not adjacent to $\phi(w) = v$, which implies  $\phi(v) \notin \{v_1, \dots, v_k\}$. Thus $\phi(v) = w.$

Recall that $N(u) \setminus \{w\} = \{v_1,\ldots,v_k\}$ is a set of $k$ vertices all adjacent to $v$. By the properties of automorphisms, it follows that $N(\phi(u)) \setminus \{\phi(w)\} = N(v_i)\setminus \{v\}$ is a set of $k$ vertices all adjacent to $\phi(v)=w$. Therefore $N(v_i) \setminus \{v\}$ must consist entirely of shadow vertices.

Now, by construction, $N(v_i)$ is equally split between original vertices and their corresponding shadow vertices. Since $v$ is the only original vertex in $N(v_i)$, we can conclude that $N(v_i) = \{u,v\}$, so $d(v_i)=2$. Since $\phi(u)=v_i$ by assumption, $d(u) = 2$ as well. This gives our desired contradiction.\end{proof}

Lemma~\ref{lem:phiwisv} leaves only two possibilities for automorphisms that do not fix the root. One is that $|G| < 3$. For example, $\mu(K_2) = C_{5}$, which is vertex-transitive.

The other way an automorphism might not fix $w$ is to map it to a shadow vertex. For example, Figure~\ref{fig:muK13x2} shows $\mu(K_{1,3})$ with original vertices in black, shadow vertices in orange, and the root in white. The vertical reflectional symmetry of this drawing induces an automorphism that moves the root to a shadow vertex. Such an automorphism exists for every star graph $K_{1,m}$ with $m \geq 0$. We show in Lemma~\ref{lem:phiwisu} that star graphs are the only graphs in which the root is not fixed by every automorphism of $\mu(G)$.

Before our next lemma, we introduce the following definition and notation. 

\begin{defn}
Given a vertex $v$ in a graph, let the \emph{neighborhood degree multiset of $v$}, denoted $D_v$, be $\{d(u): u\in N(v)\}$.
\end{defn}

Properties of automorphisms guarantee for every vertex $v$ and automorphism $\phi$, that $D_v = D_{\phi(v)}$. We use this fact in the proof of Lemma~\ref{lem:phiwisu} and in the proofs in Section~\ref{sec:gen}.

\begin{lemma} \label{lem:phiwisu}
 If there is an automorphism $\phi$ of $\mu(G)$ that takes the root $w$ to a shadow vertex, then $G = K_{1,m}$ for some $m \geq 0$. Additionally, if $|G| \neq 2$, then $\phi(w)$ is the shadow vertex of the unique vertex of maximum degree in $G$.
 \end{lemma}
 
\begin{proof}
Let $\phi$ be an automorphism of $\mu(G)$ such that $\phi(w)$ is a shadow vertex. Let $|G|=n$ and label the vertices of $\mu(G)$ so that $\phi(w) = u_n$.

If $n = 1$, then $G = K_{1,0}$, and $\mu(G)$ has independent vertex $v_1$ together with a $K_2$ consisting of shadow vertex $u_1$ and root $w$. Clearly $\phi(w)$ must be $u_1$, the only other nonisolated vertex in $\mu(G)$.

Suppose $n > 1$. Since $\phi(w) = u_n$, by properties of automorphisms, $D_{w}=D_{\phi(w)} = D_{u_n}$. We show this equality guarantees $G = K_{1,n{-}1}$.

By construction of the Mycielskian, we have $N(w) = \{u_1,\ldots, u_n\}$ and $d(u_i)=d_G(v_i){+}1$. Thus \[D_{w} = \{d_G(v_1){+}1,\ldots,d_G(v_n){+}1\}.\] 

By construction and properties of graph automorphisms $n = d(w) = d(u_n)$. Then, since $u_n$ is not adjacent to $v_n$, it must be that $N(u_n) = \{v_1,\ldots,v_{n{-}1},w\}$. Since $d(v_i)=2d_G(v_i)$, we see that \[D_{u_n} = \{2d_G(v_1),\ldots, 2d_G(v_{n{-}1}),d(w)\}.\]

 With $D_{w} = D_{u_n}$ we have 

\[\{d_G(v_1){+}1,\ldots,d_G(v_n){+}1\} = \{2d_G(v_1),\ldots, 2d_G(v_{n{-}1}),d(w)\}.\]

Recall $d(w) = n = d(u_n)$ and by construction $d(u_n) = d_G(v_n)+1$, so removing $d(w)=d_G(v_n)+1$ yields

\begin{equation}\label{eqn:D_Ns} \{d_G(v_1){+}1,\ldots,d_G(v_{n{-}1}){+}1\} = \{ 2d_G(v_1),\ldots, 2d_G(v_{n{-}1})\}. \end{equation}

We will now show this is impossible when $G\neq K_{1,m}$ for $m \ge 1$. We have already that $d_G(v_n) = n{-}1$, so suppose that for some value of $i$ with $1\leq i\leq n{-}1$, we have $d_G(v_i)>1$. Define
\[ d_{\min}=\min_{1 \leq i \leq n{-}1} \{ d_G(v_i) : d_G(v_i)>1 \}. \] 
Let $j\in \{1,\ldots , n{-}1\}$ be such that $d_G(v_j)=d_{\min}>1$.

Then in Equation~\ref{eqn:D_Ns} on the left hand side $d_G(v_j)+1$ is the smallest value greater than $2$, and on the right hand side, $2d_G(v_j)$ is the smallest value greater than $2$. Thus $d_G(v_j)+1 = 2d_G(v_j)$. However, this can only hold if $d_G(v_j) = 1$, a contradiction of $d_G(v_j)>1$.

Therefore, we must have $d_G(v_i)=1$ for $1 \leq i \leq n{-}1$ and $d_G(v_n)=n{-}1$. Thus, $G = K_{1,n{-}1}$ for some $n\geq 2$. 
Furthermore, if $|G|\geq 3$ then $v_n$ is the unique vertex of maximum degree in $G$, and $\phi(w)$ is its shadow.\end{proof}

\section{Generalized Mycielskian Graphs} \label{sec:gen}

In this section, we define and examine generalized Mycielskian graphs and their automorphisms. The organizational structure and results mirror those in Section~\ref{sec:autos}, although the proofs have some differences. 

The generalized Mycielskian of $G$, also known as a cone over $G$, was introduced by Stiebitz~\cite{stiebitz1985beitrage} in 1985 (cited in \cite{ct2001}) and independently by Van Ngoc~\cite{van1987graph} in 1987 (cited in~\cite{van19954}). For a fixed $t\geq 1$ and graph $G$ with vertices $\{v_1,\ldots,v_n\}$, the \emph{generalized Mycielskian of $G$}, written $\gm$, has vertices 
\[ \{u^0_1,\ldots,u^0_n,u^1_1,\ldots,u^1_n,\dots,u^t_1,\ldots,u^t_n,w\}. \]

For each edge $v_i v_j$ in $G$, the graph $\gm$ has edges $u_i^0 u_j^0$ and $u^s_i,u^{s+1}_j$, $u^s_j,u^{s+1}_i$, for $0\leq s \leq t-1$. In addition, $\gm$ has edges $u^t_i w$ for $1 \leq i \leq n$. Thus, $\gm$ has an isomorphic copy of $G$ on vertices $\{u^0_1,\ldots,u^0_n\}$, so we say $u_i^0=v_i$ for $1 \le i \le n$. We say that vertex $u_i^s$ is at level $s$; the vertices at level $0$ are called \emph{original vertices}, and the vertices at level $s\geq 1$ are called \emph{shadow vertices (at level $s$)}. The vertex $w$ is still referred to as the root, but note $w$ is only adjacent to the shadow vertices at level $t$.

In Figure~\ref{fig:GenMyExamples}, we illustrate both the traditional Mycielskian ($t=1$) and generalized Mycielskian with $t = 2$, for each of $K_2$ and $K_3$. Since $\mu_1(G) = \mu(G)$, when $t=1$ we drop the subscript for ease of notation. As before, when subscripts are omitted in degree or neighborhood notation, we are referring to degree or neighborhood in $\gm$.

\begin{figure}[htb]
\centering
 \begin{tikzpicture}[scale=.7]
\draw[black!100,line width=1.5pt] (-1,0) -- (1,0); 
\draw[fill=black!100,line width=1] (-1,0) circle (.2);
\draw[fill=black!100,line width=1] (1,0) circle (.2);
\begin{scope}[shift={(6,0)},rotate=180]
\draw[black!100,line width=1.5pt] (-1,0) -- (1,0); 
\draw[black!100,line width=1.5pt] (-1,0) -- (1,-1);
\draw[black!100,line width=1.5pt] (-1,-1) -- (1,0);
\draw[black!100,line width=1.5pt] (-1,-1) -- (0,-1.5);
\draw[black!100,line width=1.5pt] (1,-1) -- (0,-1.5);
\draw[fill=black!100,line width=1] (-1,0) circle (.2);
\draw[fill=black!100,line width=1] (1,0) circle (.2);
\draw[fill=orange!100,line width=1] (-1,-1) circle (.2);
\draw[fill=orange!100,line width=1] (1,-1) circle (.2);
\draw[fill=white!100,line width=1] (0,-1.5) circle (.2);
\end{scope}
\begin{scope}[shift={(12,0)},rotate=180]
\draw[black!100,line width=1.5pt] (-1,0) -- (1,0); 
\draw[black!100,line width=1.5pt] (-1,0) -- (1,-1);
\draw[black!100,line width=1.5pt] (-1,-1) -- (1,0);
\draw[black!100,line width=1.5pt] (-1,-1) -- (1,-2);
\draw[black!100,line width=1.5pt] (-1,-2) -- (1,-1);
\draw[fill=black!100,line width=1] (-1,0) circle (.2);
\draw[black!100,line width=1.5pt] (-1,-2) -- (0,-2.5);
\draw[black!100,line width=1.5pt] (1,-2) -- (0,-2.5);
\draw[fill=black!100,line width=1] (1,0) circle (.2);
\draw[fill=orange!100,line width=1] (-1,-1) circle (.2);
\draw[fill=orange!100,line width=1] (1,-1) circle (.2);
\draw[fill=yellow!100,line width=1] (-1,-2) circle (.2);
\draw[fill=yellow!100,line width=1] (1,-2) circle (.2);
\draw[fill=white!100,line width=1] (0,-2.5) circle (.2);
\end{scope}

\begin{scope}[shift={(0,-4)}]
\draw[black!100,line width=1.5pt] (-30:2.5) -- (-150:2.5);
\draw[black!100,line width=1.5pt] (90:2.5) -- (-150:2.5);
\draw[black!100,line width=1.5pt] (-30:2.5) -- (90:2.5);
 \draw[fill=black!100,line width=1] (-30:2.5) circle (.2);
 \draw[fill=black!100,line width=1] (90:2.5) circle (.2);
 \draw[fill=black!100,line width=1] (-150:2.5) circle (.2);
\end{scope}

\begin{scope}[shift={(6,-4)}]
\draw[black!100,line width=1.5pt] (-30:2.5) -- (-150:2.5);
\draw[black!100,line width=1.5pt] (90:2.5) -- (-150:2.5);
\draw[black!100,line width=1.5pt] (-30:2.5) -- (90:2.5);
\draw[black!100,line width=1.5pt] (-30:2.5) -- (-150:1.25);
\draw[black!100,line width=1.5pt] (90:2.5) -- (-150:1.25);
\draw[black!100,line width=1.5pt] (-30:2.5) -- (90:1.25);
\draw[black!100,line width=1.5pt] (-30:1.25) -- (-150:2.5);
\draw[black!100,line width=1.5pt] (90:1.25) -- (-150:2.5);
\draw[black!100,line width=1.5pt] (-30:1.25) -- (90:2.5);
\draw[black!100,line width=1.5pt] (-30:1.25) -- (0,0);
\draw[black!100,line width=1.5pt] (90:1.25) -- (0,0);
\draw[black!100,line width=1.5pt] (-150:1.25) -- (0,0);
 \draw[fill=black!100,line width=1] (-30:2.5) circle (.2);
 \draw[fill=black!100,line width=1] (90:2.5) circle (.2);
 \draw[fill=black!100,line width=1] (-150:2.5) circle (.2);
 \draw[fill=orange!100,line width=1] (-30:1.25) circle (.2);
 \draw[fill=orange!100,line width=1] (90:1.25) circle (.2);
 \draw[fill=orange!100,line width=1] (-150:1.25) circle (.2);
 \draw[fill=white!100,line width=1] (0,0) circle (.2);
\end{scope}

\begin{scope}[shift={(12,-4)}]
\draw[black!100,line width=1.5pt] (-30:2.5) -- (-150:2.5);
\draw[black!100,line width=1.5pt] (90:2.5) -- (-150:2.5);
\draw[black!100,line width=1.5pt] (-30:2.5) -- (90:2.5);
\draw[black!100,line width=1pt] (-30:2.5) -- (-150:1.66);
\draw[black!100,line width=1pt] (90:2.5) -- (-150:1.66);
\draw[black!100,line width=1pt] (-30:2.5) -- (90:1.66);
\draw[black!100,line width=1pt] (-30:1.66) -- (-150:2.5);
\draw[black!100,line width=1pt] (90:1.66) -- (-150:2.5);
\draw[black!100,line width=1pt] (-30:1.66) -- (90:2.5);
\draw[black!100,line width=1pt] (-30:.83) -- (-150:1.66);
\draw[black!100,line width=1pt] (90:.83) -- (-150:1.66);
\draw[black!100,line width=1pt] (-30:.83) -- (90:1.66);
\draw[black!100,line width=1pt] (-30:1.66) -- (-150:.83);
\draw[black!100,line width=1pt] (90:1.66) -- (-150:.83);
\draw[black!100,line width=1pt] (-30:1.66) -- (90:.83);
\draw[black!100,line width=1pt] (-30:.83) -- (0,0);
\draw[black!100,line width=1pt] (90:.83) -- (0,0);
\draw[black!100,line width=1pt] (-150:.83) -- (0,0);
 \draw[fill=black!100,line width=1] (-30:2.5) circle (.175);
 \draw[fill=black!100,line width=1] (90:2.5) circle (.175);
 \draw[fill=black!100,line width=1] (-150:2.5) circle (.175);
 \draw[fill=yellow!100,line width=1] (-30:.83) circle (.175);
 \draw[fill=yellow!100,line width=1] (90:.83) circle (.175);
 \draw[fill=yellow!100,line width=1] (-150:.83) circle (.175);
  \draw[fill=orange!100,line width=1] (-30:1.66) circle (.175);
 \draw[fill=orange!100,line width=1] (90:1.66) circle (.175);
 \draw[fill=orange!100,line width=1] (-150:1.66) circle (.175);
 \draw[fill=white!100,line width=1] (0,0) circle (.175);
\end{scope}
\end{tikzpicture}
\caption{Top: $K_2$, $\mu(K_2)$ and $\mu_2(K_2)$, drawn with vertical levels with the root on the top. Bottom: $K_3$, $\mu(K_3)$ and $\mu_2(K_3)$, drawn with concentric levels with the root in the middle.}
\label{fig:GenMyExamples}
\end{figure}
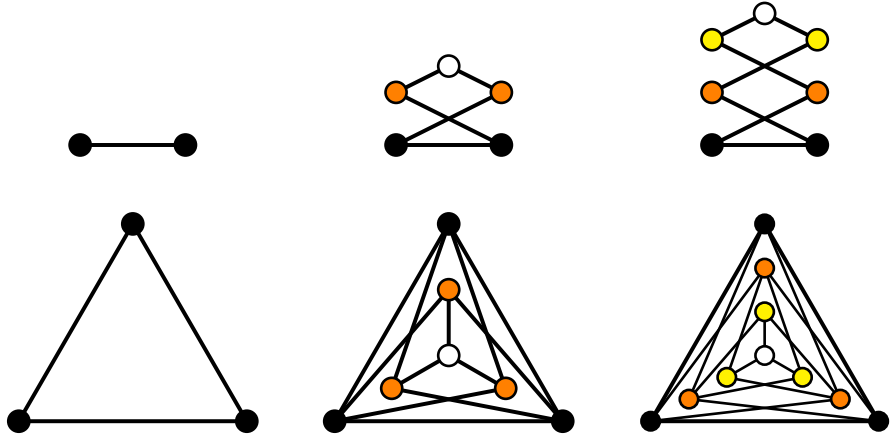

\medskip\noindent{\bf \boldmath Facts about $\gm$:}  Let $|G|=n$, $t\geq 1$, and $d_G(v_i)=k$. The generalized Mycielski construction gives us the following: $|\gm|=(t+1)n+1$; $d(w) = n$; for $0\leq s \leq t-1$, $d(u^s_i) = 2k$; $d(u^t_i)= k+1$; for $1\leq s\leq t$, the set of shadow vertices at level $s$ is independent.

The results in Section \ref{sec:autos} for the traditional Mycielskian of a graph correspond closely to many of the results for the generalized Mycielskian. To indicate as much, we have labeled appropriate extended results in the same manner as in Section \ref{sec:autos}, only with a prime added. The exception is Lemma~\ref{lem:disconnected}, which is only needed for the generalized Mycielskian. As in the case for $\mu(G)$, to prove results about automorphisms of $\gm$, we consider cases based on the image of the root. The following lemma shows that if $G$ is disconnected, then every automorphism of $\gm$ fixes the root.

\begin{lemma}\label{lem:disconnected} If $G$ is a disconnected graph and $\phi$ is an automorphism of $\gm$, then $\phi$ maps the root $w$ to itself. \end{lemma}

\begin{proof}
We show here that under the given hypotheses, $w$ is the only vertex of $\gm$ whose removal increases the number of connected components. That is, $w$ is the only cut-vertex in $\gm$. Since every graph automorphism must preserve properties of connectedness, every automorphism of $\gm$ must, therefore, map $w$ to itself. 

First, consider the deletion of $w$. Let $v_i$ and $v_j$ be in distinct components of $G$. By the Mycielski construction, $u_i^t$ and $u_j^t$ are both adjacent to $w$ in $\gm$. However, $u_i^t, w, u_j^t$ is the only path between $u_i^t$ and $u_j^t$ and so in $\gm\setminus\{w\}$, we have $u_i^t$ and $u_j^t$ in distinct components. This shows that $\gm\setminus\{w\}$ has more components than $\gm$ and so $w$ is a cut-vertex in $\gm$.

We now consider the deletion of other vertices in $\gm$, all of the form $u_i^s$ for $0 \leq s \leq t$, such that $u_i^0 = v_i$ is either an isolated or a nonisolated vertex in $G$. We show their deletion from $\gm$ does not increase the number of components.

Consider the vertex $u_i^s$ for $0\leq s \leq t$ such that $u_i^0 $ is a nonisolated vertex in $G$. For each neighbor $v_j$ of $v_i$ in $G$, the following cycle exists in $\gm$: $v_i,u^1_j,u^2_i,\dots,w,\ldots,u_j^2,u_i^1,v_j,v_i$. Observe that this cycle contains $u_i^s$ and, further, that every neighbor of $u_i^s$ is contained in a cycle of this form. Thus, removing $u_i^s$ from $\gm$ does not disconnect the graph. Hence, $\gm\setminus\{u_i^s\}$ has the same number of components as $\gm$ and $u_i^s$ is not a cut-vertex.

Finally, consider the vertex $u_i^s$ for $0\leq s \leq t$ such that $u_i^0$ is an isolated vertex in $G$. If $s \neq t$, then $u_i^s$ is isolated in $\gm$ and so $u_i^s$ cannot be a cut-vertex. If $s = t$, then $u_i^t$ has $w$ as its only neighbor and is also not a cut-vertex.

It follows that $w$ is the only cut-vertex in $\gm$ and so every automorphism of $\gm$ must fix $w$.\end{proof}

Knowing that any automorphism of a disconnected graph fixes the root allows us in many cases to only consider connected graphs $G$. The following lemma also provides us with a useful structural property. In particular, if $G$ is a graph with at least three vertices and $\gm$ has an automorphism mapping $w$ to an original vertex or a shadow vertex not at level $t$, then $G$ does not have a dominating vertex.

\begin{manual2}[lemma]{\ref{lem:nodomvertex}\thmprime}
\label{lem:GMnodomvertex}
Let $G$ be a graph with $|G|\geq 3$ and 
$t\geq 1$. Let $\phi$ be an automorphism of $\gm$. If $\phi(w)$ is a vertex at level $s$ for $0\leq s\leq t{-}1$, then $G$ does not have a dominating vertex.
\end{manual2}

\begin{proof}
Let $|G|=n\geq 3$. Assume that $\phi$ is an automorphism of $\gm$ with $\phi(w)$ either an original vertex or a shadow vertex at level $s$ for some $1\leq s\leq t{-}1$. Label the vertices of $G$ so that $\phi(w) = u_n^s$ and so that $N_G(v_n) = \{v_1,\dots, v_k\}$, where $k =d_G(v_n)$.
If $s = 0$, then $u_n^s = v_n$.

By properties of automorphisms and the generalized Mycielskian construction, $d(\phi(w)) = d(u_n^s) =2k$ and $d(\phi(w))= d(w) = n$. Thus, $n=2k$. With $n\geq 3$ it follows that $k\ge 2$. Since $d_G(v_n)=k = \frac{d(u^s_n)}{2}$ by construction, $d_G(v_n) = \frac{n}{2}$. Further, since $k\ge 2$ and $n\ge 3$, we get $\frac{n}{2}\neq n{-}1$, so $v_n$ is not a dominating vertex in $G$. It follows that any dominating vertex in $G$ must be in $N(v_n)$. 

Suppose there exists $j\in \{1,\ldots, k\}$ so that $v_j$ is a dominating vertex of $G$. Then $d_G(v_j)=n{-}1$, so $d(v_j) = 2(n{-}1)$. If $s\ge 1$, by construction $d(u_j^{s{-}1}) = 2(n{-}1)$. Also, since $v_j\in N(v_n)$, if $s \ge 1$, we have $u_j^{s{-}1}\in N(u_n^s)$, and if $s=0$, we have $v_j\in N(u_n^0)$. Thus, for any $0 \le s \le t{-}1$, we have a  degree $2n{-}2$ vertex adjacent to $\phi(w)=u_n^s$. By properties of automorphisms, this implies that $w$ has a neighbor of degree $2n{-}2$. However, by construction, all neighbors of $w$ have degree at most $n$. Since $n<2n{-}2$ for $n\ge 3$, we achieve a contradiction. 

Hence, if $\phi(w)$ is a vertex at level $s$ for some $0\leq s\leq t{-}1$, then $G$ does not have a dominating vertex.\end{proof}

We will now show that for $|G| \geq 3$, any automorphism of $\gm$ that does not fix the root $w$, must map $w$ to a shadow vertex at level $t$. Note that Lemma~\ref{lem:phiwisv} addresses the case that $t=1$. 

\begin{manual2}[lemma]{\ref{lem:phiwisv}\thmprime}\label{lem:phiwisv'} 
Let $G$ be a graph with $|G|  = n \ge 3$ and $t >1$. Then no automorphism of $\gm$ maps the root $w$  to $u^s_i$, for any $1\leq i \leq n$, $0 \leq s\leq t{-}1$.
\end{manual2}

\begin{proof}
Let $G$ be a graph with $|G|=n \ge 3$. By Lemma \ref{lem:disconnected}, if $G$ is disconnected, then every automorphism  $\phi$ of $\gm$ satisfies $\phi(w)=w$. Thus, we need only consider the case when $G$ is connected. 

Suppose there is an automorphism $\phi$ of $\gm$ that maps the root $w$ to $u_i^s$ for some $1\leq i\leq n$, $0\leq s\leq t{-}1$. Label the vertices so that $\phi(w) = u_n^s$ and $N_G(v_n) = \{v_1,\ldots,v_k\}$, meaning $d_G(v_n) = k$. 

We split the remainder of the proof into cases: $s=t{-}1$ and $0\leq s\leq t{-}2$.

\medskip

\noindent{\bf Case\,(I):} Suppose that $\phi(w) = u_n^{t{-}1}$. By construction and since automorphisms preserve degrees, $2k = d(u_n^{t-1}) = d(w) = n$. We will show that there is no possible image for $u_n^{t{-}1}$ under $\phi$.

Since $u_n^{t{-}1}$ is distance 2 from $w$, by properties of automorphisms, $\phi(u_n^{t-1})$ is distance two from $\phi(w)=u_n^{t-1}$. To see the choices for $\phi(u_n^{t-1})$, we need only look at the endpoints of paths of length two from $\phi(w) = u_n^{t{-}1}$. Recall that shadow vertices at levels $s\in \{1,\ldots , t\}$ are independent sets. Thus, unless $t=2$, a path of length 2 from $u_n^{t{-}1}$ must change levels at each vertex. Thus such paths can only take one of the following forms: $u_n^{t{-}1} u_i^t w$, $u_n^{t{-}1} u_i^{t} u_j^{t{-}1}$, $u_n^{t{-}1} u_i^{t{-}2} u_j^{t{-}1}$, $u_n^{t{-}1} u_i^{t{-}2} u_j^{t{-}3}$, or $u_n^1 v_i v_j$, where the latter two paths require $t \ge 3$ and $t=2$, respectively.
Thus, we consider three subcases: Case\,(Ia): $\phi(u_n^{t{-}1})=w$; Case\,(Ib): $\phi(u_n^{t{-}1})=u^{t{-}1}_j$ for $1 \le j \le n{-}1$; Case\,(Ic): either $t\ge 3$ and $\phi(u_n^{t{-}1})=u^{t{-}3}_j$ or $t=2$ and $\phi(u_n^{t{-}1})=v_j$ for $1 \le j \le n$. 

\smallskip

\noindent{\textbf{Case\,(Ia):}} Suppose that $\phi(u_n^{t{-}1})=w$. First, suppose that $t\geq 3$. Since $u_n^{t{-}3}$ has distance 2 from $u_n^{t{-}1}$, and since automorphisms preserve distances, we must have that $\phi(u_n^{t-3})$ has distance 2 from $\phi(u_n^{t-1})=w$. Thus, since $\phi(w) = u_n^{t{-}1}$, we have $\phi(u_n^{t-3}) =u_j^{t{-}1}$, for some $1 \leq j \leq n{-}1$. We will show that in fact $\phi(u_n^{t-3}) = u_j^{t-1}$ for some $1\leq j \leq k$. We then will show this implies that $G$ has a dominating vertex, contradicting Lemma~\ref{lem:GMnodomvertex}. If $t=2$, then $u_n^0$ is distance 2 from $u_n^{t{-}1}$, and the following argument still holds, replacing $u_n^{t{-}3}$ with $u_n^0$.

First, since $N_G(v_n) = \{v_1, \dots, v_n\}$, by construction the common neighbors of $u_n^{t-1}$ and $u_n^{t-3}$ are $u_1^{t-2},\ldots,u_k^{t-2}$.

Thus, these $k$ vertices must be mapped to common neighbors of $\phi(u_n^{t-1}) = w$ and $\phi(u_n^{t-3}) = u_j^{t-1}$. However, the common neighbors of $w$ and $u_j^{t-1}$ are the neighbors of $u_j^{t-1}$ at level $t$. Thus, $u_j^{t-1}$ has exactly $k$ neighbors at level $t$. 
Since these are disjoint from the neighbors of $u_n^{t{-}1}$ at level $t$ and $n = 2k$, they must be $\{ u_{k+1}^{t},\dots,u_n^t \}$. Hence, by construction, $N_G(v_j) =\{v_{k+1}, \ldots, v_n\}$. In particular, $v_j\in N_G(v_n)$ and so $1 \le j \le k$.

We have already shown that $d(u_j^{t{-}1})=2k$, so $d(u_j^{t{-}2}) = 2k$ as well. Further, since $v_j\in N(v_n)$ we see that $u_j^{t{-}2} \in N(u_n^{t{-}1}) = N(\phi(w))$. Therefore, by properties of automorphisms, $w$ also has a neighbor of degree $2k$, say $u_i^t$. By construction this implies $d_G(v_i) = 2k{-}1 = n{-}1$, so  $v_i$ is dominating vertex in $G$. 
This contradicts Lemma~\ref{lem:GMnodomvertex}.

Thus $\phi(u_n^{t{-}1})\ne w$.

\smallskip

\noindent{\textbf{Case\,(Ib):} }
Suppose that $\phi(u_n^{t{-}1}) =u^{t{-}1}_j$, for some $1 \leq j \leq n{-}1$. Note, that since $\phi(w) = u_n^{t{-}1}$, we cannot have $j = n$. We will show that $d_G(v_n)=k$ must be both even and odd, a contradiction. First we will show that $\phi(u_n^t)=u_{i}^{t{-}2}$, for some $1 \leq i \leq k$. 

Since $u_n^t$ is adjacent to $w$, $\phi(u_n^t)$ is adjacent to $\phi(w)=u_n^{t{-}1}$. Hence, $\phi(u_n^t)$ must be a neighbor of $u_n^{t{-}1}$ at level $t$ or at level $t{-}2$. So, $\phi(u_n^t) \in \{u_1^t, \dots, u_k^t, u_1^{t{-}2},\dots u_k^{t{-}2}\}$. We next show that $\phi(u_n^t) = u^t_i$ for any $1 \le i \le k$ leads to a contradiction so that $\phi(u_n^t)=u_{i}^{t{-}2}$ for some $1 \leq i \leq k$.

By construction $d(u_n^t) = k+1$, so as automorphisms preserve degrees, if $\phi(u_n^t) = u^{t}_i$, then $d(u_i^t)=k+1$. Then, by construction, $d(u_i^{t{-}2})=2k$. Moreover, with $1 \le i \le k$, we have $u^{t{-}2}_i\in N(u_n^{t{-}1}) = N(\phi(w))$. So $w$ must also be adjacent to a vertex of degree $2k$, say $u_j^t$. By construction, since $u_j^t$ is a top-level shadow vertex, $d(u_j^t) = d_G(v_j)+1$, so $d_G(v_j) =2k{-}1 = n{-}1$. Thus, $v_j$ is a dominating vertex in $G$ contradicting Lemma~\ref{lem:GMnodomvertex}. Hence, $\phi(u_n^t)\ne u_i^t$ for any $1 \leq i \leq k$ and therefore, $\phi(u_n^t) = u_i^{t{-}2}$ for some $1 \leq i \leq k$.

Then, since $d_G(v_n) = k$, we have $k+1=d(u_n^t)= d(\phi(u_n^t))=d(u^{t-2}_i)$. However, since $u_i^{t-2}$ is not a top-level shadow vertex, by construction we also have that $d(u^{t-2}_i)=2d_G(v_i)$. With $k+1 = 2d_G(v_i)$,  $k$ must be odd.

Since $u_n^{t{-}1}$ and $w$ have $k$ common neighbors, namely $u^t_1, \ldots, u_k^t$, we see that $\phi(u_n^{t-1})=u_j^{t-1}$ and $\phi(w)=u_n^{t-1}$ must have $k$ common neighbors as well. Since $u_j^{t{-}1}$ and $u_n^{t{-}1}$ are at the same level, by construction, their common neighbors must be split evenly between vertices at level $t$ and vertices at level $t{-}2$. This implies that $k$ is even, a contradiction with our earlier conclusion that $k$ is odd. 

Thus $\phi(u_n^{t-1}) \ne u^{t-1}_j$ for any $1 \leq j \leq n-1$.

\smallskip

\noindent \textbf{Case\,(Ic):} 
Suppose that either $t\ge 3$ and $\phi(u_n^{t{-}1})=u^{t{-}3}_j$ or $t=2$ and $\phi(u_n^{t{-}1})=v_j=u_j^0$ for some $1 \le j \le n$. Say $\phi(u_n^{t{-}1}) = u_j^r$ with $r \in \{t{-}3,0\}$. Note that if $t=3$, then $r=0=t-3$. 

Since automorphisms preserve degrees, the neighborhood degree multisets $D_{u_n^{t{-}1}}$ and $D_{u_j^{r}}$ are equal. This will yield a contradiction similar to the one in Lemma~\ref{lem:phiwisu}.

By construction half the neighbors of $u_n^{t{-}1}$ are at level $t$ with degree $d_G(v_i)+1$ for each $1\le i \le k$ and half are at level $t{-}2$ with degree $2d_G(v_i)$ for each $1 \le i \le k$. Thus the neighborhood degree multiset of $u_n^{t{-}1}$ is
\begin{equation}
D_{u_n^{t{-}1}}= \{ d_G(v_1){+}1,\dots,d_G(v_k){+}1,2d_G(v_1),\dots,2d_G(v_k)\}. \label{eqn:Dunt-1}
\end{equation} 

By construction, if $t \ge 4$, then   $r = t{-}3 > 0$ and so a vertex at level $r$ has half its neighbors at level $t-4$ and the other half at level $t-2$. If $t=2$ or $t=3$, then $r=0$ and so a vertex at level $r$ has half its neighbors at level $0$ and the other half at level $1$. Thus the neighbors of $\phi(u_n^{t-1}) = u^{r}_j$ are not at level $t$, and therefore have degree $2d_G(v_i)$ for some $1 \le i \le n$. 

To be more precise about $N(u_j^r)$, let $d_G(v_j) = \ell$ and write $N_G(v_j)=\{v_{i_1},\ldots, v_{i_\ell}\}$ for appropriate indices $i_j$. By construction, if $t\geq 4$, we have $N(u_j^{r}) = N(u_j^{t-3})= \{u^{t-4}_{i_1},\ldots, u^{t-4}_{i_\ell}, u^{t-2}_{i_1},\ldots, u^{t-2}_{i_\ell} \}$ and therefore 
\begin{equation}
D_{u_j^{r}} = \{2d_G(v_{i_1}),\ldots, 2d_G(v_{i_\ell}),2d_G(v_{i_1}),\ldots, 2d_G(v_{i_\ell})\}. \label{eqn:Dujr}
\end{equation} If $t=3$ or $t=2$ so that $r=0$, levels $t{-}4$ and $t{-}2$ above get replaced by levels 0 and 1 in $N(u_j^r)$. This yields the same degree multiset as in Equation~\ref{eqn:Dujr}. 

Thus, equality of the multisets $D_{u_n^{t-1}}$ and $D_{u_j^r}$ gives equality of the sets in Equations~\ref{eqn:Dunt-1} and~\ref{eqn:Dujr}. We can conclude that $\ell = k$. Furthermore, we have $2d_G(v_1) = 2d_G(v_{i_j})$ for some $i_j\in \{i_1,\ldots, i_\ell\}$. Proceeding inductively, we can reindex $\{1,\ldots,\ell\}$ if necessary so that $d_G(v_j) = d_G(v_{i_j})$. Thus, dropping these identical elements from each set, and using the equality gained from reindexing, we get:

\[ \{ d_G(v_1){+}1,\dots,d_G(v_k){+}1\} = \{2d_G(v_1),\ldots, 2d_G(v_k)\}. \] 

Using the same argument used in the proof of Lemma~\ref{lem:phiwisu}, we see that this is only possible if all $k$ neighbors of $v_n$ have degree 1 in $G$. However, if all neighbors of $v_n$ in $G$ have degree 1, then our assumption that $G$ is connected requires that $G$ be a star graph and that $v$ be dominating in $G$. This contradicts Lemma~\ref{lem:GMnodomvertex}. 

We conclude then that if $t \ge 3$, then $\phi(u_n^{t{-}1}) \ne u^{t{-}3}_j$ and if $t=2$, then $\phi(u_n^{t{-}1}) \ne u_j^0$, for any $1\leq j\leq n$.

\smallskip

This finishes Case\,(I), so that $\phi(w) \neq u_n^{t-1}$. 

\medskip

\noindent{\bf Case\,(II):} Suppose that $\phi(w)=u_n^s$ for some $0\leq s\leq t{-}2$.

Since $d_G(v)=k$, and $s < t$, we have $d(u_n^s)=2k$. Hence $\phi(w) = u_n^s$ gives $d(u_n^s) = d(w) = n$. We will show the equality $D_{w}=D_{\phi(w)} = D_{u^s_n}$ required by properties of automorphisms leads to a contradiction. 

By construction, $N(w) = \{u^t_1,\ldots, u^t_n\}$ and $d(u^t_i)=d_G(v_i)+1$. Thus \[D_{w} = \{d_G(v_1){+}1,\ldots,d_G(v_n){+}1\}.\] 

If $1 \le s \le t{-}2$, then $N(u^s_n)=\{u^{s{-}1}_1,\ldots, u^{s{-}1}_{k},u^{s{+}1}_1,\ldots, u^{s{+}1}_{k}\}$, and since $d(u^{s{+}1}_i) = d(u^{s{-}1}_i)=2d_G(v_i)$, we see that 
\[ D_{u^s_n} = \{2d_G(v_1), 2d_G(v_1),\ldots, 2d_G(v_k), 2d_G(v_{k})\}.\]
If $s=0$, level $s{-}1$ above gets replaced by level 0 in $N(u_n^s)$. This gives the same neighborhood degree multiset for $D_{u^s_n}$.

Thus equality of $D_{w}$ and $D_{u^s_n}$ gives
\begin{equation} \{d_G(v_1){+}1,\ldots,d_G(v_{n}){+}1\} = \{2d_G(v_1), 2d_G(v_1),\ldots, 2d_G(v_k), 2d_G(v_{k})\}. \label{eqn:NwNus}\end{equation}

If there exists an $i$ in $1 \le i \le k$ with $d_G(v_k) > 1$, let
\[ d_{\min}=\min_{1 \leq i \leq k} \{ d_G(v_i) : d_G(v_i)>1 \}. \] 
Let $j$ in be such that $d_G(v_j) = d_{\min}$. Then as $d_G(v_j)+1$ appears on the left hand side of Equation~\ref{eqn:NwNus}, there is a $j'$ with $1 \le j' \le k$ such that $d_G(v_j)+1=2d_G(v_{j'})$. Because $d_G(v_j) = d_{\min}>1$, we find $d_G(v_{j'})>1$. Then, by selection of $d_{\min}$ and that $1 \le j' \le k$, we have $d_G(v_j) \le d_G(v_{j'})$. However, the equality $d_G(v_j)+1=2d_G(v_{j'})$ and $d_G(v_j) > 1$ imply that $d_G(v_{j'}) < d_G(v_j)$. This contradiction lets us conclude $d_G(v_i)=1$ for all $1 \le i \le k$. 

But, then all neighbors of $v_n$ in $G$ have degree 1. Thus, our assumption that $G$ is connected requires that $G$ be a star graph with $v_n$ be dominating in $G$. This contradicts Lemma~\ref{lem:GMnodomvertex}. 

Thus $\phi(w)\ne u_n^s$ for any $0\leq s\leq t-2$.\end{proof}

Lemma~\ref{lem:phiwisv'} leaves two possibilities for automorphisms that do not fix the root. One is that $|G| < 3$. Of these, $K_1$ is fully addressed  by Lemma~\ref{lem:phiwisu'}. Additionally, $K_1+K_1$, where $+$ indicates disjoint union, is a disconnected graph, which is addressed by Lemma~\ref{lem:disconnected}. Finally, we have $K_2 = K_{1,1}$. Here, we have $\mu_t(K_2) = C_{2t+3}$, a vertex-transitive graph. As we will see in Lemma~\ref{lem:phiwisu'}, $K_2$ is the only star graph with automorphisms not mapping $w$ to a top-level shadow vertex. 

The other possibility is that $G$ has an automorphism where $w$ is mapped to a top-level shadow vertex. Lemma~\ref{lem:phiwisu'} shows that this only occurs when $G$ is a star graph. 

\begin{manual2}[lemma]{\ref{lem:phiwisu}\thmprime}\label{lem:phiwisu'} 
 If there is an automorphism $\phi$ of $\gm$ that takes the root $w$ to a shadow vertex at level $t$, 
 then $G = K_{1,m}$ for some $m \geq 0$. Additionally, if $|G| \ne 2$, then $\phi(w)$ is the shadow vertex at level $t$ of the unique vertex of maximum degree in $G$. \end{manual2}

\begin{proof}
Let $|G|=n$ and let $\phi$ be an automorphism of $\gm$ such that $\phi(w)$ is a shadow vertex at level $t$. Label the vertices so that $\phi(w) = u_n^t$. Then $u_n^t$ is a shadow of $v_n$.

Suppose $n=1$. Then $G = K_{1,0}$ and $\gm$ is a set of isolated vertices, $\{u_1^0, u_1^1, \dots, u_1^{t{-}1}\}$, together with a $K_2$ consisting of shadow vertex $u_1^t$ and root $w$. Clearly $\phi(w)$ must be $u_1^t$, the only other nonisolated vertex of $\gm$.

Now, suppose $n > 1$. Since $u_n^t = \phi(w)$, we have $D_{u_n^t} = D_{w}$. As in Lemma~\ref{lem:phiwisu}, this allows us to conclude $G = K_{1,n{-}1}$. 

By construction and properties of automorphisms, $n = d(w) = d(u_n^t) = d_G(v_n){+}1$. Thus, $d_G(v_n) = n-1$, so that $N_G(v_n) = \{v_1,\ldots,v_{n{-}1}\}$. Hence, $N(u_n^t) = \{ u_1^{t{-}1},\dots, u_{n{-}1}^{t{-}1},w\}$ and
\[ D_{u_n^t} = \{d(u_1^{t{-}1}),\dots, d(u_{n{-}1}^{t{-}1}),d(w)\} = \{ 2d_G(v_1),\dots, 2d_G(v_{n-1}), d(w)\}.\]

On the other hand, by construction $N(w) = \{u_1^t,\ldots,u_n^t\}$. Thus, 
\[ D_{w}=\{d(u_1^{t}),\dots, d(u_{n}^{t})\} = \{d_G(v_1){+}1,\dots, d_G(v_n){+}1\}.\]

In $D_{w}$, we have $d(w) = n$ and in $D_{u_n^t}$ we have $d(v_n)+1 = n$, so after equating the two and removing $d(w) = d_G(v_n)+1$, we get: 

\[\{d_G(v_1){+}1,\dots, d_G(v_{n-1}){+}1\} = \{2d_G(v_1),\dots, 2d_G(v_{n-1})\}.\]

This is the same equation as Equation~\ref{eqn:D_Ns}. Hence, as in the proof of Lemma~\ref{lem:phiwisu}, we can conclude that $G=K_{1,n{-}1}$. Then $v_n$ is the unique vertex of maximum degree in $G$ and $\phi(w)$ is its shadow at level $t$.\end{proof}

\section{Distinguishing Mycielskian Graphs} \label{sec:dist}
In Sections~\ref{sec:autos} and~\ref{sec:gen} we studied the action of an automorphism on $\gm$.
For convenience in the proof of Theorem~\ref{thm:A&Sconj}, we combine Lemmas~\ref{lem:phiwisv},~\ref{lem:phiwisv'},~\ref{lem:phiwisu},~\ref{lem:phiwisu'}, and~\ref{lem:disconnected}, with the earlier observation about $K_2$, into a single lemma.

\begin{lemma} \label{lem:master}
Let $G$ be a graph and let $t \geq 1$. Let $\phi$ be an automorphism of $\gm$.
\begin{itemize}

  \item If $G = K_{1,1} = K_2$, then $\gm = C_{2t+3}$, and $\phi(w)$ can be any vertex. 
 \item If $G=K_{1,m}$ for $m \ne 1$ then $\phi(w) \in \{w,u^t\}$, where $u^t$ is the top-level shadow vertex of the vertex of degree $m$ in $K_{1,m}$.
  \item If $G \neq K_{1,m}$ for any $m$, then $\phi(w) = w$.\end{itemize}
\end{lemma}

We are now ready to state and prove our main result which says that with few exceptions, $\dist(\gm) \leq \dist(G)$. This proves Conjecture~\ref{conj:A&S} in~\cite{AS2018}.

\begin{theorem}\label{thm:A&Sconj}
Let $G$ be a graph with $\ell \ge 0$ isolated vertices and let $t\geq 1$. \begin{itemize}
  \item If $G = K_1$, then $\dist(\mu(G)) = 2$, while for $t > 1$, $\dist(\gm) = t$, exceeding $\dist(G) = 1$ for all $t$.
  
  \item If $G = K_2$, then $\dist(\mu(G)) = 3$, while for $t>1$, $\dist(\gm) = 2$, exceeding $\dist(G)=2$ only for $t=1$. 
  
  \item If $t\ell > \dist(G)$, then $\dist(\gm) = t \ell$, exceeding $\dist(G)$.

  \item Otherwise, if $G \neq K_1, K_2$ and $t\ell \le \dist(G)$, then $\dist(\gm) \le  \dist(G)$. 
\end{itemize}
\end{theorem}

Note that the last case covers nearly all graphs. For example, it covers all connected graphs with at least three vertices.

\begin{proof}
If $G = K_1$ then $\dist(G) = 1$ and $\dist(\mu(G)) = \dist(K_1+K_2)= 2$. When $t > 1$, since $G$ has $\ell=1$ isolated vertices, we have $t=t\ell>\dist(G)$, and so this case is handled below. 

If $G = K_2$, then $\dist(G)=2$. As already observed,  $\gm = C_{2t+3}$. Since $\dist(C_5)=3$ and $\dist(C_n) = 2$ when $n\geq 6$, the result holds. 

Let $|G|=n$ and $G$ have $0 \le \ell \le n$ isolated vertices. 

If $\ell > 0$, label the graph so that the isolated vertices are $v_1,\ldots,v_\ell$. By the generalized Mycielskian construction, $\gm$ has a collection of $t \ell$ mutual twins $T = \bigcup_{i=0}^{t{-}1} \{u_1^i,\ldots,u_\ell^i\}$ consisting of isolated vertices and a set of $\ell$ mutual twins $U= \{u_1^t,\ldots,u_\ell^t\}$ consisting of degree-1 neighbors of $w$. For each $0 \le s \le t$, let $R_s$ be the remaining vertices at level $s$, so that $R_s=\{u_{\ell+1}^s,\ldots,u_n^s\}$. Note if $\ell = n$, then $R_s$ is empty for each $0 \le s \le t$. Similarly, if $\ell = 0$, let $T$ and $U$ be empty. 

Suppose $t\ell>\dist(G)$. If $\ell=0$, then $t\ell =0 < \dist(G)$, so we may assume $1 \le \ell \le n$. 
Since mutual twins must receive distinct colors in a distinguishing coloring, $\dist(\gm)\geq |T| = t\ell$. We will now describe a $t\ell$-distinguishing coloring. 

First, give each vertex in $T$ a distinct color. For the vertices in $U$, give $u^t_i$ the color of $u^0_i=v_i$ for $1\leq i \leq \ell$. Next, if $\ell < n$, use at most $\dist(G)<t\ell$ colors on $\{v_{\ell+1}, \dots, v_{n}\}=R_0$ so that the induced coloring on $G$ is distinguishing and also color each shadow vertex $u_j^s$ the same color as $v_j$ for $1\leq s \leq t$ and $\ell<j\leq n$. Finally, give $w$ any of the $t\ell$ colors, other than the color on $v_1$ and $u_1^t$.

Now, let $\phi$ be an automorphism of $\gm$ that respects this coloring. If $G=K_1$, then $w$ and $u_1^t$ having different colors means $\phi$ fixes $w$. Otherwise, the presence of isolated vertices means that $G$ is not $K_{1,m}$ for any $m\geq 1$ and so, by Lemma~\ref{lem:master}, $\phi$ fixes $w$. Every vertex of $T$ is fixed since these are the only vertices of degree $0$ and each has a distinct color. Similarly, the vertices in $U$ are the only vertices adjacent to $w$ with degree $1$, and each vertex of $U$ has a different color, so $\phi$ fixes each vertex in $U$. 

If $R_s$ is nonempty, then by construction, for $0 \le s \le t$, the distance between $w$ and vertices in $R_s$ is $t-s+1$. Since automorphisms preserve distances, the sets $R_0, \ldots,R_t$ are preserved by $\phi$. Since the coloring of $G$ is distinguishing, $\phi$ fixes each vertex in $G$, and therefore in the set $\{v_{\ell+1},\dots,v_n\} = R_0$. Suppose now that $R_{s-1}$ is fixed pointwise. Since for $\ell+1 \le i \le n$, we have $\phi(u_i^s) \in R_s$, let $\phi(u_i^s) = u_j^s$ for $i \neq j$. Since automorphisms preserve adjacency and $R_{s-1}$ is fixed pointwise, we have $N(u_i^s) \cap R_{s-1} = N(\phi(u_i^s)) \cap R_{s-1} = N(u_j^s) \cap R_{s-1}$. By construction, this can only occur if $v_i$ and $v_j$ are twins. However, since the coloring restricted to $G$ is distinguishing, $v_i$ and $v_j$ have different colors. Thus, in our coloring $u_i^s$ and $u_j^s$ received different colors, a contradiction. This shows that $R_s$ must be fixed pointwise as well.

Thus, $\phi$ fixes every vertex of $\gm$ and so we have $t\ell$-distinguishing coloring of $\gm$. This shows that when $t\ell > \dist(G),$ we have $\dist(\gm) = t\ell$.

\medskip

For the remainder of the proof, we assume $G\neq K_1, K_2$ and $t\ell\leq \dist(G)$. We consider two cases based on whether the automorphism fixes the root. 

\smallskip 

Suppose first that $\gm$ has an automorphism that does not fix $w$. Since $G \neq K_1, K_2$, by Lemma~\ref{lem:master}, $G=K_{1,m}$ for some $m\geq 2$. Hence, $\dist(G) = m$. Let $v_{m+1}$ be the unique vertex of degree $m$ in $G$. By the structure of $K_{1,m}$ and $\mu_t(K_{1,m})$, we see that 
\begin{itemize} 
	\item $d(u^s_{m+1}) = 2m$ for all $0\leq s \leq t-1$ and $d(u^t_{m+1}) = m+1$; 
	\item vertices $v_1,\ldots,v_m$ are mutually twin in $\gm$ since each has neighborhood $\{v_{m+1} , u^1_{m+1}\}$;
	\item for each $1 \le s \le t$, vertices $u^s_1,\ldots, u^s_m$ are mutually twin in $\gm$ with shared neighborhood $\{u_{m+1}^{s{-}1},u_{m+1}^{s+1}\}$ when $s \neq t$ and $\{u_{m+1}^{t{-}1},w\}$ when $s=t$.
	\end{itemize} 

Note that since $v_1,\ldots, v_m$ are mutually twin, each needs a distinct color in a distinguishing coloring. Therefore, $\dist(\gm) \geq m$. We claim that, in fact, $\dist(\gm) = m$.

Consider the following $m$-coloring of $\gm$: for $1 \le i \le m$ assign color $i$ to
 $u^s_i$, for $0 \leq s \leq t$. Assign color $1$ to $w$ and color $2$ to $u^s_{m+1}$, for $0 \leq s \leq t$. Suppose that $\phi$ is an automorphism of $\gm$ that preserves these color classes. Let $C_i$ be the set of vertices with color $i$.
 
 We have $C_1=\{u_1^0, u^1_1,\dots, u^t_1,w\}$. Since each  vertex in $C_1\setminus\{w\}$ has degree 2, while $w$ has degree $m+1> 2$, we have $w$ is fixed by $\phi$. Furthermore, since the distance from $w$ to $u_1^s$ is $t-s+1$, these unique distances from a vertex fixed by $\phi$ guarantee that $C_1$ is fixed pointwise by $\phi$.

We have $C_2=\{u^0_2, \ldots, u^t_2, u^0_{m+1},\ldots, u^t_{m+1} \}$. The vertices in $\{u^0_2,\dots, u^t_2 \}$ have degree $2$, while the vertices in $\{u^0_{m+1},\dots, u^{t}_{m+1}\}$ have degree $2m$ or $m+1$, each of which is strictly greater than $2$. Therefore, $\phi$ fixes each setwise. Furthermore, as before, within each of these subsets, the vertices have distinct distances from the fixed vertex $w$. Thus, $C_2$ is also fixed pointwise by $\phi$.

For each $3\leq i \leq m$, we have $C_i=\{ v_i,u^1_i,\dots, u^t_i\}$. Again, the vertices of $C_i$ have distinct distances from the fixed vertex $w$, and so $C_i$ is fixed pointwise by $\phi$.

Thus, this is an $m$-distinguishing coloring of $\mu_t(K_{1,m})$ when $m \ge 2$ so that $\dist(\mu_t(K_{1,m})) = \dist(K_{1,m})$ for $m \ge 2$. In particular, when $G \neq K_1, K_2$, $t \ell \le \dist (G)$, and $G$ has an automorphism that does not fix $w$, we have $\dist(\gm) \le \dist(G)$. 

Finally, suppose that every automorphism of $\gm$ fixes $w$. Recall that we have assumed $G\ne K_1,K_2$ and that $t\ell \leq \dist(G)$. Let $\dist(G)=k$ and fix a $k$-distinguishing coloring of $G$. We extend this coloring to a $k$-distinguishing coloring of $\gm$.

First, color all original vertices in $\gm$ with the $k$-distinguishing coloring of $G$. To be distinguishing, any twin vertices in $G$ must receive different colors. In particular, if $\ell \ge 2$, the isolated vertices of $G$ have distinct colors. As before, extend the coloring to the rest of the isolated vertices in $T$, giving each a distinct color. Since $|T| = t \ell \le \dist(G)$, we have enough colors for this step. For vertices that are not isolated, color each shadow vertex $u^s_j$ the same color as $u_j^0=v_j$, for $1\leq s \leq t$, $1 \leq j \leq n$. Finally, give $w$ any of the $k$ colors. We claim this is an $k$-distinguishing coloring of $\gm$.

Let $\phi$ be an automorphism of $\gm$ that respects this coloring. Since all vertices of $T$ received different colors, $\phi$ fixes all isolated vertices. For $0 \le s \le t$, the sets $R_s= \{u_{\ell+1}^s,\ldots,u_n^s\}$ are nonempty. As before, the distance between vertices in $R_s$ and $w$ is a function of $s$. Since $w$ is fixed,  these sets are preserved setwise by $\phi$. Also as before, our coloring of $R_0$ comes from a distinguishing coloring of $G$, so $R_0$ is fixed pointwise. An induction argument can again be used to show that this guarantees each set $R_s$ is fixed pointwise, so that we have a distinguishing coloring of $\gm$. 

Thus, $\dist(\gm)\leq k = \dist(G)$ when $w$ is fixed and $t\ell \le \dist(G)$.\end{proof}

The following corollary is immediate from Theorem \ref{thm:A&Sconj} since if $G$ has $\ell$ isolated vertices then $\dist(G)\geq t\ell$ when $t=1$. The corollary proves and exceeds the conjecture by Alikhani and Soltani.
\begin{cor} For all graphs $G$ with $G\neq K_1, K_2$,
 $\dist(\mu(G))\leq \dist(G)$.
\end{cor}

In summary, for traditional Mycielskian graphs, the only exceptions are $K_1$ and $K_2$. We note that $K_2$ is an unsurprising exception since $\mu(K_2)=C_5$ is, in a sense, an exception among cycles, since it is the only cycle with distinguishing number 3 that is realizable as a Mycielskian graph. Furthermore, we proved that for generalized Mycielskian graphs with $t > 1$, the only exception is when $\gm$ has so many isolated vertices that their number exceeds $\dist(G)$.

We note here that we have not proved that $\dist(G)=\dist(\gm)$. In fact, generalized Mycielskians of complete graphs show us that $\dist(G)$ and $\dist(\gm)$ may be arbitrarily far apart. We have $\dist(K_n) = n$ always. On the other hand, for $n \ge 3$, Proposition~\ref{prop:muKn} below  shows that $\dist(\mu(K_n)) = \lceil~\sqrt{n}~\rceil$. Additionally, if $n \ge 3$ and $t \ge \log_2 n-1$, then $\dist(\mu_t(K_n)) = 2$. Using white as color 1 and red as color 2, Figure~\ref{fig:muK3} shows the 2-distinguishing colorings described in Proposition~\ref{prop:muKn} for $\mu(K_3)$ and $\mu_2(K_3)$.

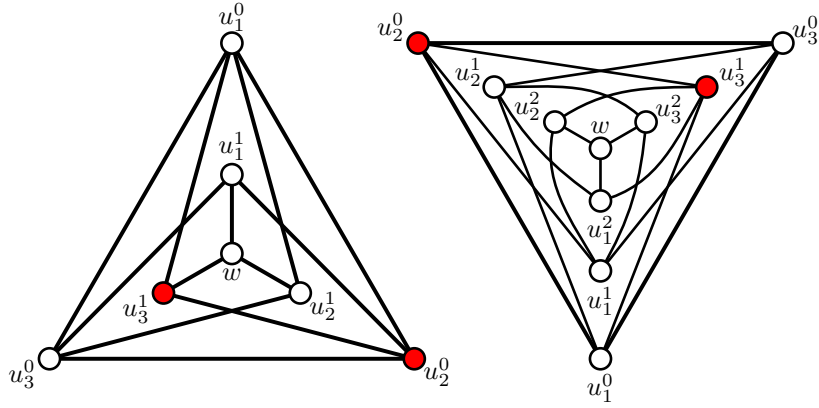
\begin{figure}[ht]
  \centering
  \begin{tikzpicture}[scale=.7]
\draw[black!100,line width=1.5pt] (-30:4) -- (-150:4);
\draw[black!100,line width=1.5pt] (-30:4) -- (90:4); 
\draw[black!100,line width=1.5pt] (90:4) -- (-150:4); 
\draw[black!100,line width=1.5pt] (-30:1.5) -- (-150:4); 
\draw[black!100,line width=1.5pt] (-30:1.5) -- (90:4); 
\draw[black!100,line width=1.5pt] (90:1.5) -- (-150:4);
\draw[black!100,line width=1.5pt] (-30:4) -- (-150:1.5); 
\draw[black!100,line width=1.5pt] (-30:4) -- (90:1.5); 
\draw[black!100,line width=1.5pt] (90:4) -- (-150:1.5);
\draw[black!100,line width=1.5pt] (0,0) -- (-150:1.5);
\draw[black!100,line width=1.5pt] (0,0) -- (90:1.5); 
\draw[black!100,line width=1.5pt] (0,0) -- (-30:1.5);
\draw[fill=red!100,line width=1] (-30:4) circle (.2);
\draw[fill=white!100,line width=1] (90:4) circle (.2);
\draw[fill=white!100,line width=1] (-150:4) circle (.2);
\draw[fill=white!100,line width=1] (-30:1.5) circle (.2);
\draw[fill=white!100,line width=1] (90:1.5) circle (.2);
\draw[fill=red!100,line width=1] (-150:1.5) circle (.2);
\draw[fill=white!100,line width=1] (0,0) circle (.2);
\draw (90:4.5) node{$u^0_1$};
\draw (-30:4.5) node{$u^0_2$};
\draw (-150:4.6) node{$u^0_3$};
\draw (90:2.05) node{$u^1_1$};
\draw (-30:2) node{$u^1_2$};
\draw (-150:2.1) node{$u^1_3$};
\draw (0,-.4) node{$w$};
\begin{scope}[shift={(7,2)},yscale=-1]
\draw[black!100,line width=1.5pt] (-30:4) -- (-150:4);
\draw[black!100,line width=1.5pt] (90:4) -- (-150:4);
\draw[black!100,line width=1.5pt] (-30:4) -- (90:4);
\draw[black!100,line width=1pt] (-30:4) -- (-150:2.33);
\draw[black!100,line width=1pt] (90:4) -- (-150:2.33);
\draw[black!100,line width=1pt] (-30:4) -- (90:2.33);
\draw[black!100,line width=1pt] (-30:2.33) -- (-150:4);
\draw[black!100,line width=1pt] (90:2.33) -- (-150:4);
\draw[black!100,line width=1pt] (-30:2.33) -- (90:4);
\draw[black!100,line width=1pt] (-30:1) to[out=-135,in=0] (-150:2.33);
\draw[black!100,line width=1pt] (90:1) to[out=215,in=60] (-150:2.33);
\draw[black!100,line width=1pt] (-150:1) to[out=-255,in=-120] (-270:2.33);
\draw[black!100,line width=1pt] (-30:1) to[out=95,in=-60] (-270:2.33);
\draw[black!100,line width=1pt] (90:1) to[out=-15,in=120] (-30:2.33);
\draw[black!100,line width=1pt] (210:1) to[out=325,in=180] (-30:2.33);
\draw[black!100,line width=1pt] (-30:1) -- (0,0);
\draw[black!100,line width=1pt] (90:1) -- (0,0);
\draw[black!100,line width=1pt] (-150:1) -- (0,0);
 \draw[fill=white!100,line width=1] (-30:4) circle (.2);
 \draw[fill=white!100,line width=1] (90:4) circle (.2);
 \draw[fill=red!100,line width=1] (-150:4) circle (.2);
 \draw[fill=white!100,line width=1] (-30:1) circle (.2);
 \draw[fill=white!100,line width=1] (90:1) circle (.2);
 \draw[fill=white!100,line width=1] (-150:1) circle (.2);
  \draw[fill=red!100,line width=1] (-30:2.33) circle (.2);
 \draw[fill=white!100,line width=1] (90:2.33) circle (.2);
 \draw[fill=white!100,line width=1] (-150:2.33) circle (.2);
 \draw[fill=white!100,line width=1] (0,0) circle (.2);
 \draw (90:4.55) node{$u^0_1$};
\draw (-30:4.5) node{$u^0_3$};
\draw (-150:4.6) node{$u^0_2$};
 \draw (90:2.9) node{$u^1_1$};
\draw (-30:2.9) node{$u^1_3$};
\draw (-150:2.9) node{$u^1_2$};
 \draw (90:1.55) node{$u^2_1$};
\draw (-30:1.5) node{$u^2_3$};
\draw (-150:1.6) node{$u^2_2$};
\draw (0,-.4) node{$w$};
\end{scope}
\end{tikzpicture}
  \caption{A 2-distinguishing coloring of $\mu(K_{3})$ and $\mu_2(K_{3}).$ 
  } \label{fig:muK3}
\end{figure}

\begin{prop}\label{prop:muKn}
Let $n \ge 3$ and $t\geq 1$. Let $k\in \mathbb{N}$ be the least value satisfying $k^{t+1} \ge n$. Then $\dist(\mu_t(K_n))=k$.
\end{prop}

\begin{proof}
Let $k$ be the least value satisfying $k^{t+1} \geq n$. Since $k^{t+1} > n{-}1$, the base-$k$ representation of $n{-}1$ has at most $t{+}1$ digits with each digit  between 0 and $k{-}1$. For each $1 \le i \le n$, let $r_i$ be the representation of $i{-}1$ in base $k$, with leading 0s appended so that $r_i$ has $t{+}1$ digits. 

We give a $k$-coloring of $\mu_t(K_n)$ as follows: give $w$ color $1$ and for $1 \le i \le n$ and $0 \le s \le t$, give $u_i^s$ color $c{+}1$ if the $(s{+}1)$-st digit in $r_i$ is $c$. Since $0 \leq c \leq k-1$, this is a $k$-coloring of $\gm$. We will prove that this $k$-coloring is distinguishing. 

Given any $i$ and $j$, with $i \ne j$ there is an $\hat{s}$ with $0 \le \hat{s} \le t$ such that $r_i$ and $r_j$ are different in digit $\hat{s}{+}1$. Therefore, at level $\hat{s}$, vertices $u_i^{\hat{s}}$ and $u_j^{\hat{s}}$ receive different colors. 

By Lemma~\ref{lem:master}, every automorphism of $\mu_t(K_n)$ fixes $w$. Since automorphisms preserve distances, the levels are fixed setwise by every automorphism. Moreover, by construction, for $1 \le s \le t$ and $1 \le i \le n$, the only non-neighbor of $u_i^s$ at level $s{-}1$ is $u_i^{s{-}1}$. Since automorphisms preserve non-adjacency, $\phi(u_i^{\hat{s}}) = u_j^{\hat{s}}$ for some $\hat{s}$ if and only if $\phi(u_i^s) = u_j^s$ for all $0 \le s \le t$. 

However, we have shown for each $i \neq j$ there exists an $\hat{s}$ where the colors on $u_i^{\hat{s}}$ and $u_j^{\hat{s}}$ differ. Thus, to preserve the color classes, an automorphism $\phi$ must have $\phi(u_i^s) = u_i^s$ for all $1 \le i \le n$ and $0 \le s \le t$. Thus, this coloring is $k$-distinguishing and so $\dist(\mu_t(K_n)) \le k$. 

Let $\ell \in \mathbb{N}$ such that $\ell < k$. Since $k$ is the least value satisfying $k^{t+1} \geq n$, it must be the case that $\ell^{t+1} < n$. We claim there does not exist an $\ell$-distinguishing coloring of $\mu_t(K_n)$.

There are at most $\ell^{t+1}$ lists of the form $(c_0,\ldots,c_t)$ with $1 \le c_s \le \ell$ for each $0 \le s \le t$. Hence, by Pigeonhole Principle, in any $\ell$-coloring of $\mu_t(K_n)$, there exist distinct $i$ and $j$ such that the colors of $u_i^s$ and $u_j^s$ agree for each $0 \le s \le t$. Then, the automorphism $\phi$ with $\phi(u_i^s) = u_j^s$ and $\phi(u_j^s) = u_i^s$ for each $0 \le s \le t$ and $\phi(x) = x$ for all other vertices $x$, preserves the color classes. Hence, there does not exist an $\ell$-distinguishing coloring of $\mu_t(K_n)$ for all $\ell < k$. It follows that $\mu_t(K_n) \geq k$ and, therefore, $\mu_t(K_n) = k$.\end{proof}

\section{Acknowledgments}

The work described in this article is a result of a collaboration made possible by the Institute for Mathematics and its Applications' Workshop for Women in Graph Theory and Applications,
August 2019.

\bibliographystyle{plain}
\bibliography{MycielskiBib}

\end{document}